\documentclass[12pt,dvips]{amsart}

\usepackage[dvips]{epsfig}
\usepackage{fancyhdr}
\usepackage{afterpage}
\usepackage{chngcntr}
\usepackage{graphicx}
\usepackage{latexsym}
\usepackage{amsfonts}
\usepackage{amscd}
\usepackage{amssymb}
\usepackage{amsmath}
\usepackage{amsthm}
\usepackage{bm}
\usepackage{psfrag}
\usepackage{url}
\usepackage{lscape}
\usepackage[usenames]{color}
\usepackage[all]{xy}
\usepackage{enumitem}

\usepackage[lmargin=1.25in,rmargin=1.25in,tmargin=1.25in,bmargin=1.25in]{geometry}

\usepackage[unicode, breaklinks=true, linktocpage, bookmarksnumbered=true,backref=true,letterpaper]{hyperref}

\newtheorem{prop}{Proposition} [section]
\newtheorem{lemma}[prop]{Lemma}
\newtheorem{thm}[prop]{Theorem}

\theoremstyle{definition}

\def\N{\mathbb{N}}
\def\R{\mathbb{R}}
\def\C{\mathbb{C}}
\def\Z{\mathbb{Z}}
\def\Q{\mathbb{Q}}

\def\CF {\widehat{\operatorname{CF}}}

\def\HFK {\widehat{\operatorname{HFK}}}

\def\HF {\widehat{\operatorname{HF}}}

\def\CC {\mathcal{C}}

\def\HH {\mathcal{H}}

\def\TT {\mathcal{T}}

\newcommand{\abs}[1] {\left\lvert #1 \right\rvert}

\def\minus{\smallsetminus}

\definecolor{darkblue}{rgb}{0,0,0.5}
\definecolor{darkred}{rgb}{0.5,0,0}
\definecolor{darkgreen}{rgb}{0,0.5,0}
\definecolor{purple}{rgb}{0.5,0,0.5}
\definecolor{teal}{rgb}{0,0.5,0.5}

\title
 [Slicing Mixed Bing--Whitehead Doubles]
 {Slicing Mixed Bing--Whitehead Doubles}

\author{Adam Simon Levine}
\address {Mathematics Department \\ Brandeis University \\ 415 South Street \\ Waltham, MA 02453}
\email {levinea@brandeis.edu}

\thanks{The author was supported by NSF grants DMS-0739392 and DMS-1004622.}

\begin{document}

\maketitle

\begin{abstract}
We show that if $K$ is any knot whose Ozsv\'ath--Szab\'o concordance invariant
$\tau(K)$ is positive, the all-positive Whitehead double of any iterated Bing
double of $K$ is topologically but not smoothly slice. We also show that the
all-positive Whitehead double of any iterated Bing double of the Hopf link
(e.g., the all-positive Whitehead double of the Borromean rings) is not
smoothly slice; it is not known whether these links are topologically slice.
\end{abstract}

\section{Introduction}

A knot in the $3$-sphere is called \emph{topologically slice} if it bounds a
locally flatly embedded disk in the $4$-ball, and \emph{smoothly slice} if the
disk can be taken to be smoothly embedded. Two knots are called (topologically
or smoothly) \emph{concordant} if they are the ends of an embedded annulus in
$S^3 \times I$; thus, a knot is slice if and only if it is concordant to the
unknot. More generally, a link is (topologically or smoothly) \emph{slice} if
it bounds a disjoint union of appropriately embedded disks. The study of
concordance --- especially regarding the relationship between the notions of
topological and smooth sliceness --- is one of the major areas of active
research in knot theory, and it is closely tied to the perplexing differences
between topological and smooth $4$-manifold theory.

Given a knot $K \subset S^3$, the \emph{(untwisted) positive and negative
Whitehead doubles} of $K$, $Wh_+(K)$ and $Wh_-(K)$, and the \emph{Bing double}
of $K$, $B(K)$, are the satellites of $K$ illustrated in Figure
\ref{fig:figure8} (for $K$ the figure-eight knot). The Whitehead doubles of a
link $L$ are obtained by doubling the individual components of $L$, with a
choice of sign for each component. In particular, we denote the all-positive
and all-negative Whitehead doubles of $L$ by $Wh_+(L)$ and $Wh_-(L)$,
respectively.

Whitehead and Bing doubling play a central role in the study of concordance. In
the topological setting, Freedman \cite{FreedmanNewTechnique, FreedmanQuinn}
proved that any Whitehead double of a knot or, more generally, a \emph{boundary
link} (a link whose components bound disjoint Seifert surfaces) is
topologically slice. Moreover, the surgery conjecture for $4$-manifolds with
arbitrary fundamental group --- the central open problem in four-dimensional
topology --- is equivalent to the conjecture that the Whitehead double of any
link whose linking numbers are all zero is freely topologically
slice.\footnote{A link $L$ is \emph{freely slice} if it bounds slice disks in
$B^4$ whose complement has free fundamental group.} This conjecture is true for
two-component links \cite{FreedmanWhitehead3} but open in general. To disprove
the surgery conjecture for manifolds with free fundamental group, it would thus
suffice to show that one such link --- e.g., a Whitehead double of the
Borromean rings --- is not (freely) topologically slice. However, all such
links have resisted all attempts to determine whether or not they are
topologically slice.

\psfrag{Wh+(K)}{$Wh_+(K)$} \psfrag{Wh-(K)}{$Wh_-(K)$} \psfrag{BD(K)}{$BD(K)$}
\psfrag{K}{$K$}

\begin{figure}
\includegraphics{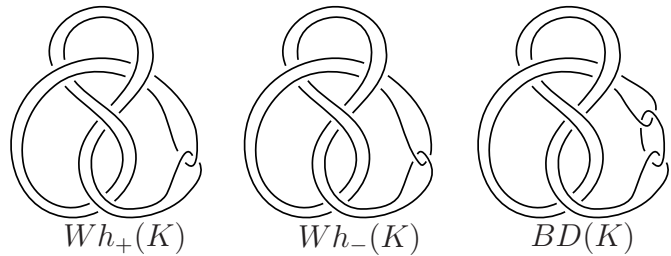}
\caption{The positive and negative Whitehead doubles and the Bing double of the
figure-eight knot.} \label{fig:figure8}
\end{figure}

Around the same time, the advent of Donaldson's gauge theory made it possible
to show that some of Freedman's examples of topologically slice knots are not
smoothly slice. Akbulut [unpublished] first proved in 1983 that the positive,
untwisted Whitehead double of the right-handed trefoil is not smoothly slice.
Later, using results of Kronheimer and Mrowka on Seiberg--Witten theory,
Rudolph \cite{RudolphObstruction} showed that any nontrivial knot that is
\emph{strongly quasipositive} cannot be smoothly slice. In particular, the
positive, untwisted Whitehead double of a strongly quasipositive knot is
strongly quasipositive; thus, by induction, any iterated positive Whitehead
double of a strongly quasipositive knot is topologically but not smoothly
slice. Bi\v{z}aca \cite{Bizaca} used this result to give explicit constructions
of exotic smooth structures on $\R^4$. Later, Hedden \cite{HeddenWhitehead}
generalized this result to any knot $K$ whose Ozsv\'ath--Szab\'o invariant
$\tau(K)$ (an integer-valued concordance invariant coming from the knot Floer
homology of $K$ \cite{OSz4Genus, RasmussenThesis}) is positive. It is
conjectured \cite[Problem 1.38]{KirbyList} that if $Wh_\pm(K)$ is smoothly
slice, then $K$ itself must also be smoothly slice.

\begin{figure}
\includegraphics[scale=0.75]{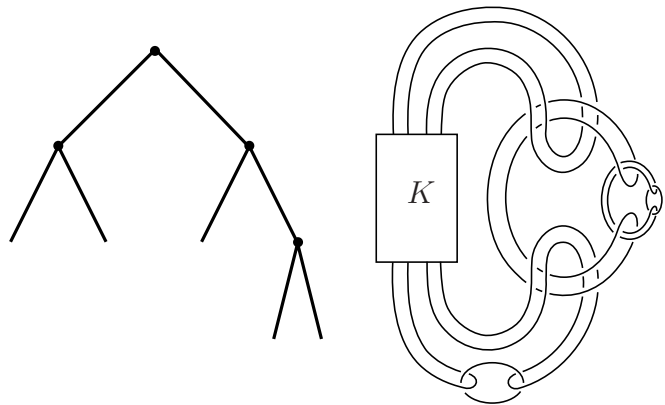}
\caption{A binary tree $T$ and the corresponding iterated Bing double
$B_T(K)$.} \label{fig:tree}
\end{figure}

We may consider \emph{partially iterated Bing doubles} of any link: at each
stage in the iteration, we replace some component by its Bing double.
Specifically, given a knot $K$, a binary tree $T$ specifies such a link
$B_T(K)$, as illustrated in Figure \ref{fig:tree}, with one component for each
leaf of $T$. For a link $L=K_1 \cup \cdots \cup K_n$ and binary trees $T_1,
\dots, T_n$, we may similarly obtain a link $B_{T_1, \dots, T_n}(L) =
B_{T_1}(K_1) \cup \cdots \cup B_{T_n}(K_n)$. In particular, if $H$ is the Hopf
link, the links obtained in this manner are known as \emph{generalized
Borromean links}, since Bing doubling one component of $H$ yields the Borromean
rings.

Using the author's work in \cite{LevineDoublingOperators} --- a lengthy
computation of $\tau$ for a particular family of satellite knots --- we shall
prove:

\begin{thm} \label{thm:notslice}
\begin{enumerate}
\item Let $K$ be a knot with $\tau(K)>0$, and let $T$ be any
binary tree. Then the all-positive Whitehead double of $B_T(K)$,
$Wh_+(B_T(K))$, is not smoothly slice.
\item Let $H = K_1 \cup K_2$ denote the Hopf link, and let $T_1,T_2$ be
binary trees. Then $Wh_+(B_{T_1,T_2}(H))$ is not smoothly slice.
\end{enumerate}
\end{thm}

Note that for any knot $K$, $B_T(K)$ is always a boundary link (see
\cite{CimasoniSlicing} for a proof), so any Whitehead double of $B_T(K)$ (with
clasps of either sign) is topologically slice. Thus, part (1) of Theorem
\ref{thm:notslice} provides a large family of links that are topologically but
not smoothly slice.\footnote{The question of when iterated Bing doubles of a
knot are slice is also quite challenging, since the classical sliceness
obstructions vanish for iterated Bing doubles. Recent papers by Cimasoni
\cite{CimasoniSlicing} and Cha--Livingston--Ruberman
\cite{ChaLivingstonRuberman}, Cha--Kim \cite{ChaKim}, and Van Cott
\cite{VanCott} show that if an iterated Bing double of $K$ is topologically
slice, then $K$ is algebraically slice; if it is smoothly slice, then
$\tau(K)=0$. Also, Cochran, Harvey, and Leidy \cite{CochranHarveyLeidyDoubling}
have used $L^2$ signatures to find algebraically slice knots with non-slice
iterated Bing doubles.} On the other hand, it is unknown whether the links in
part (2) --- the all-positive Whitehead doubles of the generalized Borromean
links --- are topologically slice. Indeed, Freedman \cite{FreedmanQuinn} showed
that the family of Whitehead doubles (with any signs) of generalized Borromean
links is ``atomic'' for the surgery problem: whether or not they are freely
topologically slice is equivalent to the surgery conjecture. Most experts
nowadays conjecture that these links are \emph{not} topologically slice, but
the problem remains unsolved after nearly twenty-five years.

The requirement that we consider all-positive Whitehead doubles is necessary
for our proof of Theorem \ref{thm:notslice}. By taking mirrors, we also see
that the all-negative Whitehead doubles of iterated Bing doubles of knots with
$\tau(K)<0$ or of generalized Borromean links are not smoothly slice, but our
method always fails when both positive and negative Whitehead doubling are
used. Indeed, all of the gauge-theoretic invariants known to date suffer from
the same asymmetry. It is still not known whether, for instance, the positive
untwisted Whitehead double of the left-handed trefoil is smoothly slice.

\subsection*{Acknowledgments}
This paper, along with \cite{LevineDoublingOperators}, made up a large portion
of the author's thesis at Columbia University. The author is grateful to his
advisor, Peter Ozsv\'ath, and the other members of his defense committee,
Robert Lipshitz, Dylan Thurston, Paul Melvin, and Denis Auroux, for their
suggestions; to Rumen Zarev, Ina Petkova, and Jen Hom for many helpful
conversations about bordered Heegaard Floer homology; and to Rob Schneiderman,
Charles Livingston, and Matthew Hedden for their suggestions regarding link
concordance questions.

\section{Definitions}

We begin by giving more precise definitions of some of the terms used in the
Introduction.

\subsection{Infection and doubling operators}
We always work with \emph{oriented} knots and links. For any knot $K \subset
S^3$, let $K^r$ denote $K$ with reversed orientation, let $\bar K$ denote the
mirror of $K$ (the image of $K$ under a reflection of $S^3$), and let $-K =
\bar K^r$. As ${K} \# {-K}$ is always smoothly slice, the concordance classes
of $K$ and $-K$ are inverses in $\CC_1$, which justifies this choice of
notation. Note that the invariants coming from Heegaard Floer homology
($\HFK(S^3,K)$, $\tau(K)$, etc.) are sensitive to mirroring but not to
reversing the orientation of a knot.

Suppose $L$ is a link in $S^3$, and $\gamma$ is an oriented curve in $S^3
\minus L$ that is unknotted in $S^3$. For any knot $K \subset S^3$ and $t \in
\Z$, we may form a new link $I_{\gamma,K,t}(L)$, the \emph{$t$-twisted
infection of $L$ by $K$ along $\gamma$}, by deleting a neighborhood of $\gamma$
and gluing in a copy of the exterior of $K$ by a map that takes a
Seifert-framed longitude of $K$ to a meridian of $\gamma$ and a meridian of $K$
to a $t$-framed longitude of $\gamma$. Since $S^3 \minus \gamma = S^1 \times
D^2$, the resulting $3$-manifold is simply $\infty$ surgery on $K$, i.e. $S^3$;
the new link $I_{\gamma,K,t}(L)$ is defined as the image of $L$. Alternately,
let $\hat K \subset D^2 \times I$ be the $(1,1)$-tangle obtained by cutting $K$
at a point, oriented from $\hat K \cap D^2 \times \{0\}$ to $\hat K \cap D^2
\times \{1\}$. If $D$ is an oriented disk in $S^3$ with boundary $\gamma$,
meeting $L$ transversely in $n$ points, we may obtain $I_{\gamma,K,t}(L)$ by
cutting open $L$ along $D$ and inserting the tangle consisting of $n$ parallel
copies of $\hat K$, following the $t$ framing. In a link diagram, a box labeled
$K,t$ in a group of parallel strands indicates $t$-twisted infection by $K$
along the boundary of a disk perpendicular to those strands. To be precise, we
adopt the following orientation convention: If the label $K,t$ is written
horizontally and right-side-up, then $\hat K$ is oriented either from bottom to
top or from left to right, depending on whether the strands meeting the box are
positioned vertically or horizontally.\footnote{We allow both types of notation
to avoid writing labels vertically.}

Given unlinked infection curves $\gamma_1, \gamma_2$, the image of $\gamma_2$
in $I_{\gamma_1, K_1, t_1}(L \cup \gamma_2)$ is again an unknot, so we may then
infect by another pair $K_2,t_2$. We obtain the same result if we infect along
$\gamma_2$ first and then $\gamma_1$. In general, given an unlink $\gamma_1,
\dots, \gamma_n$, we may infect simultaneously along all the $\gamma_i$; the
result may be denoted $I_{\gamma_1,K_1,t_1; \, \cdots; \,
\gamma_n,K_n,t_n}(L)$, and the order of the tuples $(\gamma_i,K_i,t_i)$ does
not matter.

If $P$ is a knot (or link) in the standardly embedded solid torus in $S^3$ and
$K$ is any knot, the \emph{$t$-twisted satellite of $K$ with pattern $P$},
$P(K,t)$, is defined as $I_{\gamma,K,t}(P)$, where $\gamma$ is the core of the
complementary solid torus. The knot $K$ is called the \emph{companion}. More
generally, if we have a link $L$, we may replace a component of $L$ by its
satellite with pattern $P$, working in a tubular neighborhood disjoint from the
other components.

\begin{figure}
\psfrag{B1}{$B_1$} \psfrag{B2}{$B_2$} \psfrag{B3}{$B_3$} \centering
\includegraphics{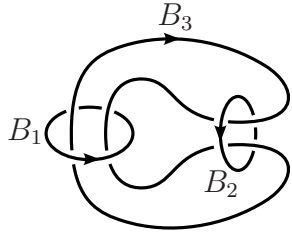}
\caption{The Borromean rings.} \label{fig:borromean}
\end{figure}

Let $B = B_1 \cup B_2 \cup B_3$ denote the Borromean rings in $S^3$, oriented
as shown in Figure \ref{fig:borromean}. The $\pm$ $t$-twisted Whitehead double
of $K$, is defined as
\[
Wh_\pm(K,t) = I_{B_1,O,\mp 1; B_2,K,t}(B_3),
\]
where $O$ denotes the unknot. (Note the sign conventions: a left-handed twist
in a pair of opposite strands is a positive clasp.) Moreover, we define the
following generalization of Whitehead doubling: for knots $J$ and $K$ and
integers $s$ and $t$, define $D_{J,s}(K,t)$ as the knot obtained from $B_3$ by
performing $s$-twisted infection by $J$ along $B_1$ and $t$-twisted infection
by $K$ along $B_2$:
\[
D_{J,s}(K,t) = I_{B_1,J,s; \, B_2,K,t}(B_3).
\]
(See Figure \ref{fig:djskt}.) The symmetries of the Borromean rings imply:
\begin{gather*}
D_{J,s}(K,t)^r = D_{J^r,s}(K,t) = D_{J,s}(K^r,t) = D_{K,t}(J,s) \\
\overline{D_{J,s}(K,t)} = D_{\bar{J},-s}(\bar{K},-t)
\end{gather*}
We also introduce the convention that when the $t$ argument is omitted, it is
taken to be zero: $D_{J,s}(K) = D_{J,s}(K,0)$.

\begin{figure}
\psfrag{J,s}[cc][cc]{{$J,s$}} \psfrag{K,t}[cc][cc]{{$K,t$}}
\includegraphics{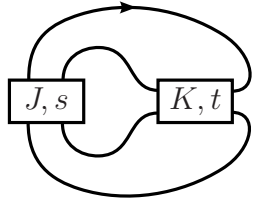}
\caption{The satellite knot $D_{J,s}(K,t)$.} \label{fig:djskt}
\end{figure}

The Bing double of $K$ may be defined as $BD(K) = I_{B_1,K,0}(B_2 \cup B_3)$;
we may also see this as a satellite operation where the pattern is a
two-component link.

\subsection{Heegaard Floer homology and the $\tau$ invariant}

In the 2000s, Ozsv\'ath and Szab\'o \cite{OSz3Manifold, OSz4Manifold}
introduced \emph{Heegaard Floer homology}, a package of invariants for $3$- and
$4$-dimensional manifolds that are conjecturally equivalent to earlier
gauge-theoretic invariants but whose construction is much more topological in
flavor. In its simplest form, given a \emph{Heegaard diagram} $\HH$ for a
$3$-manifold $Y$ (a certain combinatorial description of the manifold), the
theory assigns a chain complex $\CF(\HH)$ whose chain homotopy type is
independent of the choice of diagram; thus, the homology $\HF(Y) =
H_*(\CF(\HH))$ is an invariant of the $3$-manifold. A $4$-dimensional cobordism
between two $3$-manifolds induces a well-defined map between their Heegaard
Floer homology groups. Ozsv\'ath and Szab\'o \cite{OSzKnot} and Rasmussen
\cite{RasmussenThesis} also showed that a nulhomologous knot $K \subset Y$
induces a filtration on the chain complex of a suitably defined Heegaard
diagram, yielding an knot invariant $\HFK(Y,K)$ that is the $E^1$ page of a
spectral sequence converging to $\HF(Y)$. For knots in $S^3$, the invariant
$\HFK(S^3,K)$ categorifies the Alexander polynomial $\Delta_K$, and it is
powerful enough to detect the unknot \cite{OSzGenus} and whether or not $K$ is
fibered \cite{Ghiggini, NiFibered}.

Furthermore, the spectral sequence from $\HFK(S^3,K)$ to $\HF(S^3) \cong \Z$
provides an integer-valued concordance invariant $\tau(K)$, which yields a
lower bound on genus of smooth surfaces in the four-ball bounded by $K$:
$\abs{\tau(K)} \le g_4(K)$ \cite{OSz4Genus}. In particular, any smoothly slice
knot must have $\tau(K)=0$. Moreover, this genus bound applies not only for
surfaces in $B^4$ but for surfaces in any rational homology $4$-ball.

The $\tau$ invariant can be used to extend many of the earlier results
obstructing smooth sliceness. Hedden \cite{HeddenWhitehead} computed the value
of $\tau$ for all twisted Whitehead doubles in terms of $\tau$ of the original
knot:
\begin{equation} \label{eq:tau-Wh}
\tau(Wh_+(K,t)) = \begin{cases} 1 & t < 2 \tau(K) \\ 0 & t \ge 2 \tau(K).
\end{cases}
\end{equation}
(An analogous formula for negative Whitehead doubles follows from the fact that
$\tau(\bar K) = - \tau(K)$.) In particular, if $\tau(K)>0$, then
$\tau(Wh_+(K,0))=1$, so $Wh_+(K,0)$ (the untwisted Whitehead double of $K$) is
not smoothly slice. Since the $\tau$ invariant of a strongly quasipositive knot
is equal to its genus \cite{LivingstonComputations}, Rudolph's earlier result
follows from Hedden's.

The author's main theorem in \cite{LevineDoublingOperators} is a computation of
$\tau$ for all knots of the form $D_{J,s}(K,t)$:
\begin{thm} \label{thm:taudjskt}
Let $J$ and $K$ be knots, and let $s,t \in \Z$. Then
\[
\tau(D_{J,s}(K,t)) =
\begin{cases}
1 & s<2\tau(J) \text{ and } t<2 \tau(K) \\
-1 & s>2\tau(J) \text{ and } t>2 \tau(K) \\
0 & \text{otherwise}.
\end{cases}
\]
\end{thm}
In particular, note that if $\tau(K)>0$ and $s<2\tau(J)$, or if $\tau(K)<0$ and
$s>2\tau(J)$, then $D_{J,s}(K,0)$ is topologically slice (as its Alexander
polynomial is $1$) but not smoothly slice in any rational homology $4$-ball.

The proof of Theorem \ref{thm:taudjskt} is an involved computation using the
theory of \emph{bordered Heegaard Floer homology}, developed recently by
Lipshitz, Ozsv\'ath, and Thurston \cite{LOTBordered, LOTBimodules}. Briefly,
the bordered theory associates to a $3$-manifold with boundary a module over an
algebra associated to the boundary, so that if $Y$ is obtained by gluing
together manifolds $Y_1$ and $Y_2$ along their common boundary, the chain
complex $\CF(Y)$ may be computed as the derived tensor product of the
invariants associated to $Y_1$ and $Y_2$. If a knot $K$ is contained in, say,
$Y_1$, then we may obtain the filtration on $\CF(Y)$ corresponding to $K$ via a
filtration on the algebraic invariant of $Y_1$. This technique is thus
well-suited to the problem of computing Heegaard Floer invariants for knots
obtained through infection operations. For more details, see
\cite{LevineDoublingOperators}.

\subsection{Covering link calculus}

The proof of Theorem \ref{thm:notslice} makes use of \emph{covering link
calculus}, first developed by Cochran and Orr \cite{CochranOrrBoundary} and
used more recently by Cha and Kim \cite{ChaKim} and others
\cite{CimasoniSlicing, ChaLivingstonRuberman, VanCott}. Let $R$ denote any of
the rings $\Z$, $\Q$, or $\Z_{(p)}$ (for $p$ prime). A link $L$ in an
$R$-homology $3$-sphere $Y$ is called \emph{topologically (resp.~smoothly)
$R$-slice} if there exists a topological (resp.~smooth) $4$-manifold $X$ such
that $\partial X = Y$, $H_*(X;R)=H_*(B^4;R)$, and $L$ bounds a locally flat
(resp.~smoothly embedded), disjoint union of disks in $X$. A link that is
$\Z$-slice (in either category) is $\Z_{(p)}$-slice for all $p$, and a link
that is $\Z_{(p)}$-slice for some $p$ is $\Q$-slice. Also, a link in $S^3$ that
is slice (in $B^4$) is clearly $\Z$-slice. The key result of Ozsv\'ath and
Szab\'o \cite{OSz4Genus} is that the $\tau$ invariant of any knot that is
smoothly $\Q$-slice is $0$.

Define two moves on links in $\Z_{(p)}$-homology spheres, called \emph{covering
moves}:
\begin{enumerate}
\item Given a link $L\subset Y$, consider a sublink $L' \subset L$.
\item Given a link $L \subset Y$, choose a component $K$ with trivial
self-linking. For any $a \in \N$, the $p^a$-fold cyclic branched cover of $Y$
branched over $K$, denoted $\tilde Y$, is a $\Z_{(p)}$-homology sphere, and we
consider the preimage $L'$ of $L$ in $\tilde Y$.
\end{enumerate}
We say that $L' \subset Y'$ is a \emph{$p$-covering link} of $L \subset Y$ if
$L'$ can be obtained from $Y'$ using these moves.

The key fact is the following:
\begin{prop}
If $L$ is (topologically or smoothly) $\Z_{(p)}$-slice, then any $p$-covering
link of $L$ is also (topologically or smoothly) $\Z_{(p)}$-slice.
\end{prop}
To prove that the second covering move preserves $\Z_{(p)}$-sliceness, we take
the branched cover of the $X$ over the slice disk for $K$; the resulting
$4$-manifold is a $\Z_{(p)}$-homology $4$-ball by a well-known argument (see,
e.g., \cite[page 346]{KauffmanOnKnots}). Thus, a strategy for showing a link
$L$ is not slice is to find a knot that is a covering link of $L$ and has a
non-vanishing $\Q$-sliceness obstruction, such as $\tau$.

Note that if $L$ is a link in $S^3$ whose components are unknotted, then the
branched cover branched over one component is again $S^3$. The putative
$4$-manifold containing a slice disk, however, may change.

Henceforth, we restrict to the case where $p=2$ and omit further reference to
$p$.

\section{Proof of Theorem \ref{thm:notslice}} \label{sec:covering}

The strategy for proving the first part of Theorem \ref{thm:notslice} is to
obtain a knot $K'$ of the form
\[
K' = D_{J_1,s_1} \circ \dots \circ D_{J_n,s_n}(K),
\]
where $s_i < 2\tau(J_i)$ for each $i$, as a covering link of $Wh_+(B_T(K))$. If
$\tau(K)>0$, induction using Theorem \ref{thm:taudjskt} (which we prove below)
shows that $\tau(K')=1$, so $K'$ cannot be rationally smoothly slice, so
$Wh_+(D_T(K))$ cannot be smoothly slice. A similar argument works for the
second part of the theorem.

The following lemmas are inspired by Van Cott's work on the sliceness of
iterated Bing doubles \cite{VanCott}:

\begin{lemma} \label{lemma:solidtorus}
Let $L$ be a link in $S^3$, and suppose there is an unknotted solid torus $U
\subset S^3$ such that $L \cap U$ consists of two components $K_1,K_2$ embedded
as follows: if $A_1, A_2$ are the components of the untwisted Bing double of
the core $C$ of $U$, then $K_1 = D_{P_k,s_k} \circ \dots \circ
D_{P_1,s_1}(A_1)$ and $K_2 = D_{Q_l,t_l} \circ \dots \circ D_{Q_1,t_1}(A_2)$,
for some knots $P_1,\dots,P_k, Q_1,\dots,Q_l$ and integers $s_1, \dots, s_k,
t_1, \dots, t_l$. Let $L'$ be the link obtained from $L$ by replacing $K_1$ and
$K_2$ by the satellite knot
\begin{equation} \label{eq:solidtorus1}
C' = D_{P_k,s_k} \circ \dots \circ D_{P_1,s_1} \circ D_{R,u} (C)
\end{equation}
of $C$, where
\begin{equation} \label{eq:solidtorus2}
(R,u) = \begin{cases}
 (Q_1 \# Q_1^r, \ 2t_1) & l=1 \\
 (D_{Q_1,t_1} \circ \dots \circ D_{Q_{l-2},t_{l-2}} (D_{Q_{l-1},t_{l-1}}(Q_l \#
Q_l^r, 2t_l)), \ 0) & l>1.
\end{cases}
\end{equation}
Then $L'$ is a covering link of $L$.
\end{lemma}

\psfrag{(a)}{(a)} \psfrag{(b)}{(b)} \psfrag{(c)}{(c)} \psfrag{(d)}{(d)}
\psfrag{K1}{{\scriptsize $K_1$}} \psfrag{K2}{{\scriptsize \color{red} $K_2$}}
\psfrag{T}{{\scriptsize $T$}} \psfrag{Tb}{{\scriptsize \color{blue} $T$}}

\psfrag{Q0t0}{{\scriptsize $Q_0,t_0$}} \psfrag{Q1t1}{{\scriptsize $Q_1,t_1$}}
\psfrag{Q2t2}{{\scriptsize $Q_2,t_2$}} \psfrag{Q3t3}{{\scriptsize $Q_3,t_3$}}
\psfrag{Qltl}{{\scriptsize $Q_l,t_l$}} \psfrag{QlQl2tl}{{\scriptsize
$Q_l\#Q_l^r,2t_l$}} \psfrag{Q1Q12t1}{{\scriptsize $Q_1\#Q_1^r,2t_1$}}

\psfrag{Ql-1tl-1}[cc][cc]{{\scriptsize \shortstack[c]{$Q_{l-1},$ \\
$t_{l-1}$}}}

\psfrag{DPkskDP1s1}{{\scriptsize $[D_{P_k,s_k} \cdots D_{P_1,s_1}]$}}
\psfrag{DPkskDP1s1b}{{\scriptsize \color{blue} $[D_{P_k,s_k} \cdots
D_{P_1,s_1}]$}} \psfrag{DQltlDQ0t0}{{\scriptsize \color{red} $[D_{Q_l,t_l}
\cdots D_{Q_0,t_0}]$}}

%\psfrag{DQltlDQ1t1}{{\scriptsize \color{red} $[D_{Q_l,t_l} \cdots D_{Q_1,t_1}]$}}
\psfrag{DQltlDQ1t1}[tl][tl]{{\scriptsize \color{red} \shortstack[l]{$[D_{Q_l,t_l} \cdots $ \\
\quad $D_{Q_1,t_1}]$ }}}
\psfrag{DQltlDQ2t2}[tl][tl]{{\scriptsize \color{red} \shortstack[l]{$[D_{Q_l,t_l} \cdots $ \\
\quad $D_{Q_2,t_2}]$ }}}
\psfrag{DQltlDQ3t3}[tl][tl]{{\scriptsize \color{red} \shortstack[l]{$[D_{Q_l,t_l} \cdots $ \\
\quad $D_{Q_3,t_3}]$ }}}

\psfrag{DQltlDQ3t3oneline}{{\scriptsize \color{red} $[D_{Q_l,t_l} \cdots
D_{Q_3,t_3}]$}}

\psfrag{2l}{{\scriptsize $2^l$}} \psfrag{2l-1}{{\scriptsize $2^{l-1}$}}
\psfrag{2l-2}{{\scriptsize $2^{l-2}$}} \psfrag{2l-1-2}{{\scriptsize
$2^{l-1}-2$}}

\begin{figure} \centering
\includegraphics[scale=0.8]{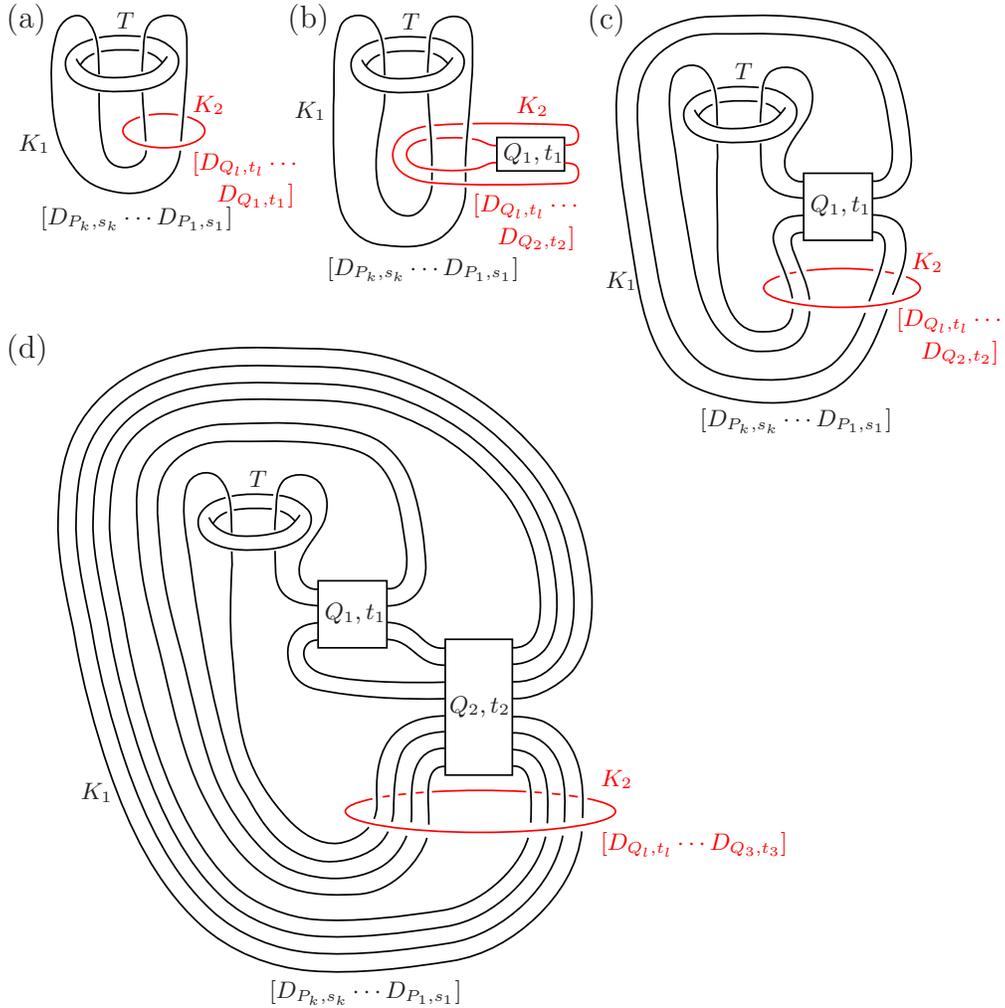}
\caption [The link described in Lemma \ref{lemma:solidtorus}.] {The link
described in Lemma \ref{lemma:solidtorus}. All but the two components shown are
contained in the interior of the solid torus $T$. We denote a satellite knot by
writing the pattern in brackets near the companion curve; thus, for instance,
$K_1=D_{P_k,s_k} \circ \cdots \circ D_{P_1,s_1}(A_1)$, where $A_1$ is the curve
shown.} \label{fig:solidtorus1}
\end{figure}

\begin{figure} \centering
\includegraphics[scale=0.8]{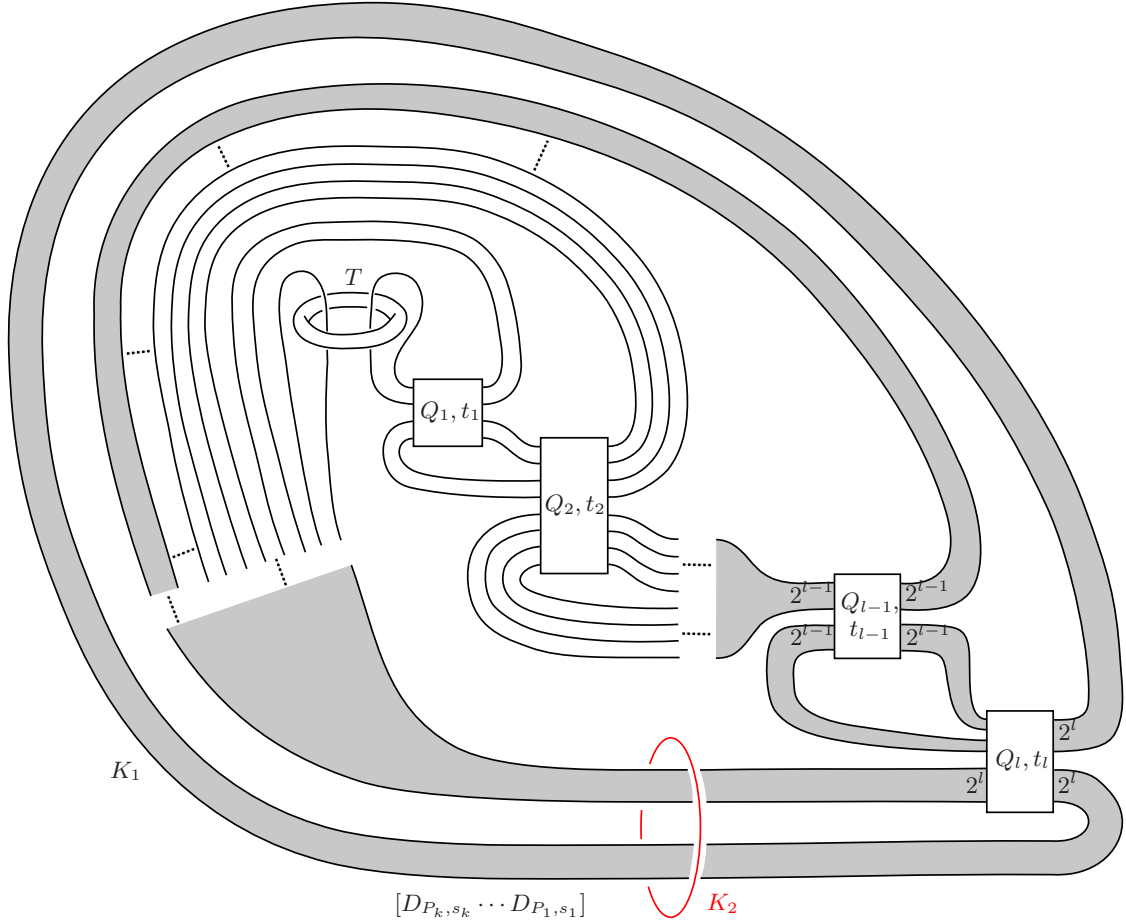}
\caption [The link described in Lemma \ref{lemma:solidtorus}, after isotopies.]
{The link described in Lemma \ref{lemma:solidtorus}, after isotopies. A shaded
region with a number represents that many parallel strands.}
\label{fig:solidtorus2}
\end{figure}

\begin{figure} \centering
\centering
\includegraphics[scale=0.79]{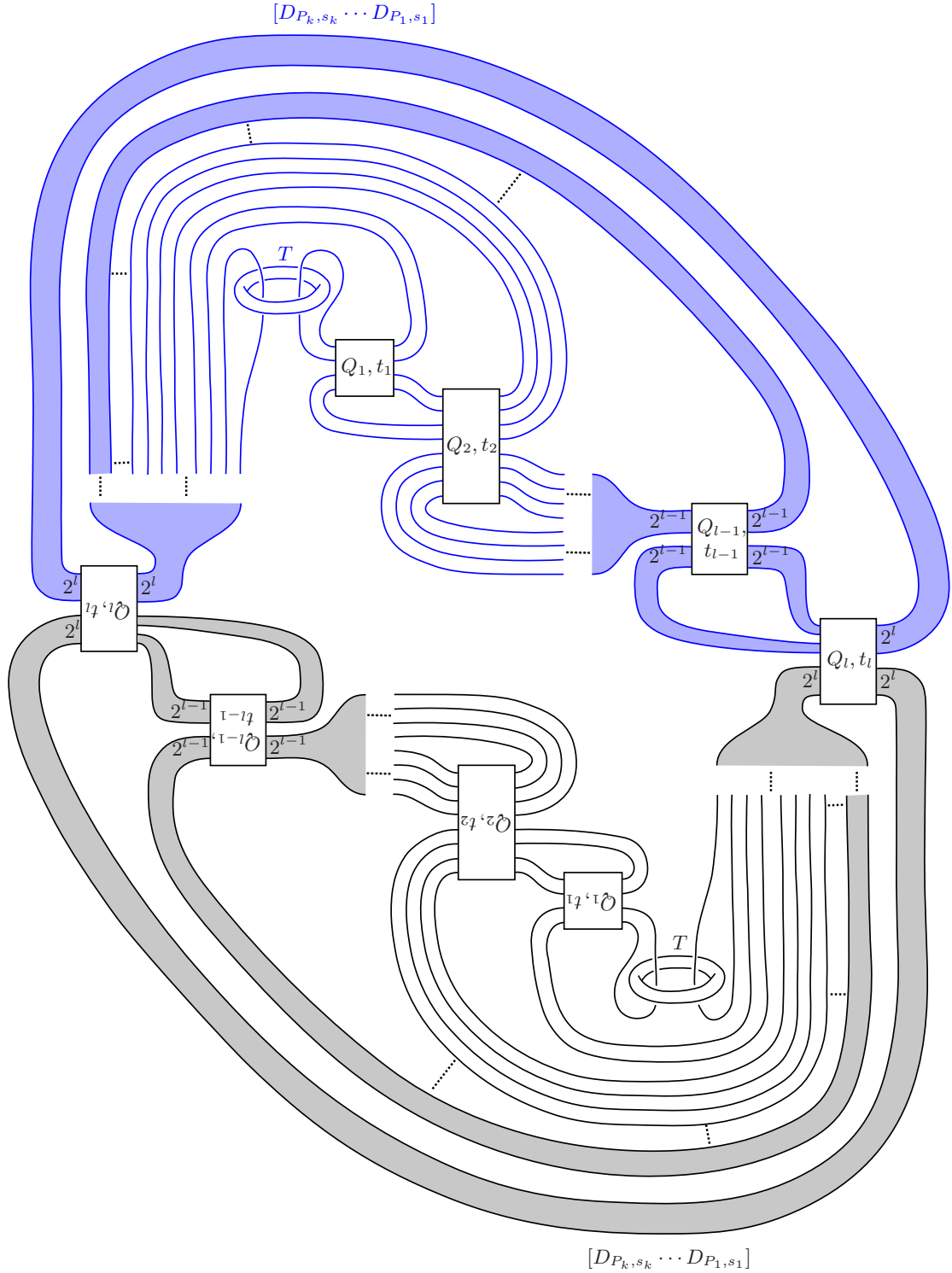}
\caption{The preimage of the link in Figure \ref{fig:solidtorus2} in the
double-branched cover of $S^3$ over $K_2$ (shown without the upstairs branch
set).} \label{fig:solidtorus3}
\end{figure}

\begin{figure} \centering
\includegraphics[scale=0.8]{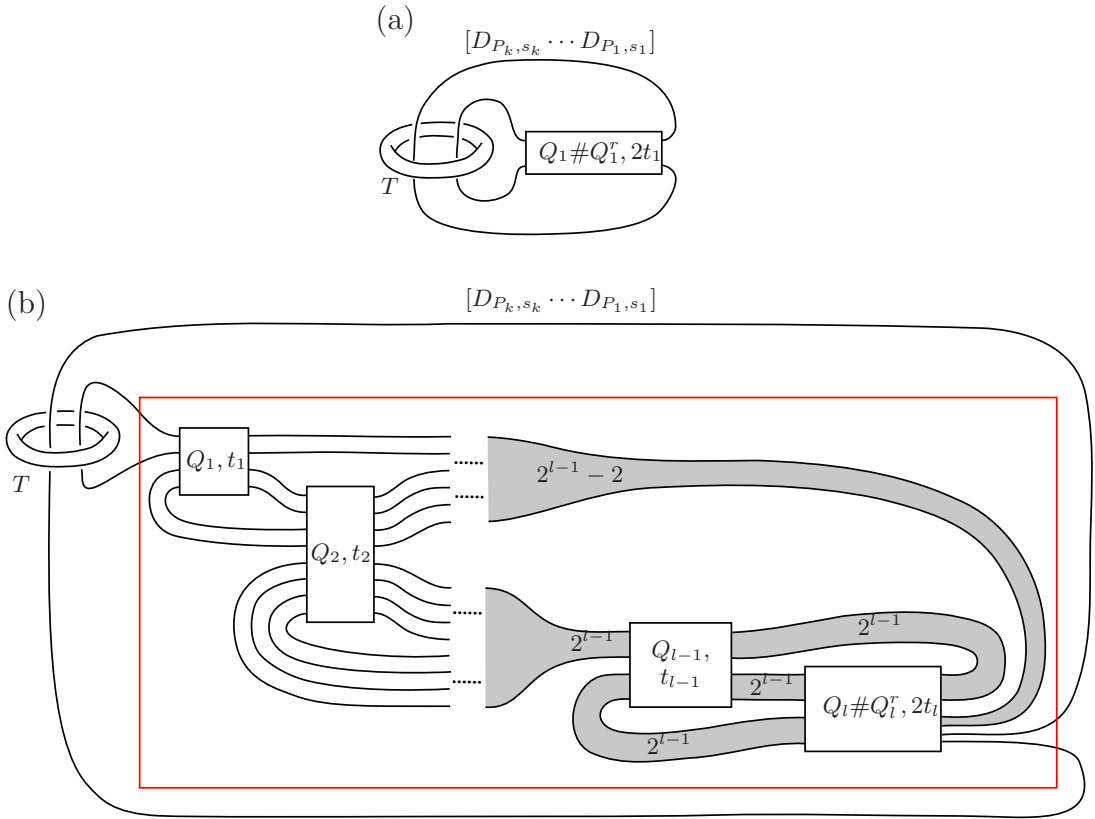}
\caption [The sublink shown in blue in Figure \ref{fig:solidtorus3}.] {The
sublink shown in blue in Figure \ref{fig:solidtorus3} is the $D_{P_k,s_k} \circ
\dots \circ D_{P_1,s_1}$ satellite of: (a) when $l=1$,
$D_{Q_1\#Q_1^r,2t_1}(C)$; (b) when $l>1$, $D_{R,0}(C)$, where $R$ is the knot
in Figure \ref{fig:solidtorus5}.} \label{fig:solidtorus4}
\end{figure}

\psfrag{DQltlDQ1t1}{{\scriptsize \color{red} $[D_{Q_l,t_l} \cdots
D_{Q_1,t_1}]$}}

\begin{figure} \centering
\includegraphics[scale=0.8]{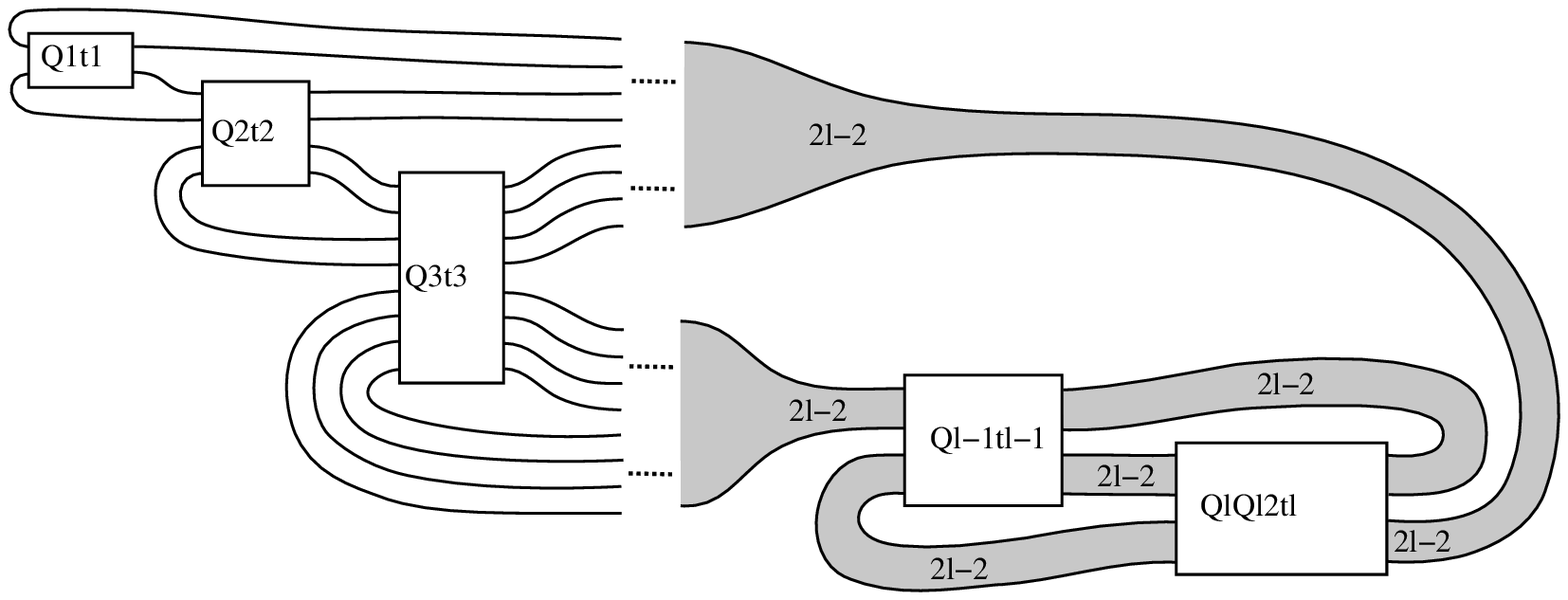}
\caption{The knot $R$ in the proof of Lemma \ref{lemma:solidtorus}.}
\label{fig:solidtorus5}
\end{figure}

\begin{proof}
Let $T = S^3 \minus U$; then $L \minus (K_1 \cup K_2)$ is contained in $T$.
Note that $K_1$ and $K_2$ are each unknotted, since $D_{J,s}(O,0) = O$ for any
$J,s$. We may untangle $K_2$ as in Figures
\ref{fig:solidtorus1}--\ref{fig:solidtorus2}. Specifically, $L$ is shown in
Figures \ref{fig:solidtorus1}(a) and (b). To obtain Figure
\ref{fig:solidtorus1}(c), we pull the two strands of the companion curve for
$K_1$ through the infection region marked $Q_1,t_1$, and then untangle the
companion curve for $K_2$. We then repeat this procedure to obtain Figure
\ref{fig:solidtorus1}(d), and $l-2$ more times to obtain Figure
\ref{fig:solidtorus2}.

The branched double cover of $S^3$ branched along $K_2$ is again $S^3$;
consider the preimage of $K_1 \cup (L \cap T)$, shown in Figure
\ref{fig:solidtorus3}. (The knot orientation conventions for infections are
important here, since the knots $Q_i$ need not be reversible.) Since $T$ is
contained in a ball disjoint from $K_1$, the sublink $L \cap T$ lifts to two
identical copies, each contained in a solid torus. The preimage of $K_2$ also
consists of two components, and each is the $D_{P_k,s_k} \circ \cdots \circ
D_{P_1,s_1}$ satellite of the companion curve shown. A sublink consisting of
one lift of each component (either the blue or the black part of Figure
\ref{fig:solidtorus3}) is redrawn in Figure \ref{fig:solidtorus4}(a) in the
case where $l=1$ and in Figure \ref{fig:solidtorus4}(b) in the case where
$l>1$. In the former case, the companion curve shown is $D_{Q_l \# Q_l^r,
2t_l}(C)$, where $C$ is the core circle of the complement of $T$. In the latter
case, it is $D_{R,0}(C)$, where we obtain $R$ by connecting the ends of one of
the two parallel strands that pass through the red box in Figure
\ref{fig:solidtorus4}(b). (A local computation shows that the linking number of
these two strands is zero, so $D_{R,0}$ is the correct operator.) The knot $R$,
shown in Figure \ref{fig:solidtorus5}, is then identified as
\[
D_{Q_1,t_1} \circ \dots \circ D_{Q_{l-2},t_{l-2}} (D_{Q_{l-1},t_{l-1}}(Q_l \#
Q_l^r, 2t_l)). \qedhere
\]
\end{proof}

\begin{lemma} \label{lemma:bing}
Let $C$ be a knot, let $U$ be a regular neighborhood of $C$, and let $A_1, A_2
\subset U$ be the components of $BD(C)$. Let $K_1 = D_{P_k,s_k} \circ \dots
\circ D_{P_1,s_1}(A_1)$ and $K_2 = D_{Q_l,t_l} \circ \dots \circ
D_{Q_1,t_1}(A_2)$, for some knots $P_1,\dots,P_k, Q_1,\dots,Q_l$ and integers
$s_1, \dots, s_k, t_1, \dots, t_l$. Let $C'$ be the knot defined by
\eqref{eq:solidtorus1} and \eqref{eq:solidtorus2}. Then $C'$ is a covering link
of $K_1 \cup K_2$.
\end{lemma}

\begin{proof}
The proof is almost identical to that of Lemma \ref{lemma:solidtorus}. The only
difference is that $S^3 \minus U$ is now a knot complement rather than a solid
torus containing some additional link components. The double branched cover
over $K_2$ contains consists of the complement of the two solid tori shown in
Figure \ref{fig:solidtorus3}, glued to two copies of $S^3 \minus U$, gluing
Seifert-framed longitude to meridian and vice versa. The resulting manifold is
again $S^3$, however. The rest of the proof proceeds \emph{mutatis mutandis}.
(Alternately, we may simply replace each of the solid tori in Figures
\ref{fig:solidtorus1}--\ref{fig:solidtorus5} by a box marked $C,0$, and proceed
as before.)
\end{proof}

A \emph{labeled binary tree} is a binary tree with each leaf labeled with a
satellite operation. Given a knot $K$ and binary tree $\TT$ with underlying
tree $T$, let $S_\TT(K)$ be the link obtained from $B_T(K)$ by replacing each
component with the satellite specified by the label of the corresponding leaf.
If $\TT$ has two adjacent leaves labeled $D_{P_k,s_k} \circ \dots \circ
D_{P_1,s_1}$ and $D_{Q_l,t_l} \circ \dots \circ D_{Q_1,t_1}$, form a new
labeled tree $\TT'$ by deleting these two leaves and labeling the new leaf
either $D_{P_k,s_k} \circ \dots \circ D_{P_1,s_1} \circ D_{Q_1 \# Q_1^r, 2t_1}
(C)$ or $D_{P_k,s_k} \circ \dots \circ D_{P_1,s_1} \circ D_{R,0}$, according to
whether $l=1$ or $l>1$, respectively, where, $R = D_{Q_1,t_1} \circ \dots \circ
D_{Q_{l-2},t_{l-2}} (D_{Q_{l-1},t_{l-1}}(Q_l \# Q_l^r, 2t_1))$ in the latter
case. We call this move a \emph{collapse}. Lemmas \ref{lemma:solidtorus} and
\ref{lemma:bing} then say that $S_{\TT'}(K)$ is a covering link of $S_\TT(K)$.

Theorem \ref{thm:taudjskt} and equations \eqref{eq:solidtorus1} and
\eqref{eq:solidtorus2}, along with the additivity of $\tau$ under connect sum,
imply:

\begin{prop} \label{prop:collapse}
Suppose $\TT'$ is obtained from $\TT$ by collapsing leaves labeled $D_{P_k,s_k}
\circ \dots \circ D_{P_1,s_1}$ and $D_{Q_l,t_l} \circ \dots \circ D_{Q_1,t_1}$,
where $s_i < 2\tau(P_i)$ and $t_i < 2\tau(Q_i)$ for all $i$. Then the label of
the new leaf of $\TT'$ has the form $D_{R_{k+1},u_{k+1}} \circ \cdots \circ
D_{R_1,u_1}$, where $u_i < 2 \tau(R_i)$.  \qed
\end{prop}

\begin{proof}[Proof of Theorem \ref{thm:notslice}]
For the first part of the theorem, note that in the new notation, $Wh_+(B_T(K))
= S_\TT(K)$, where every leaf of $\TT$ is labeled $D_{O,-1}$. Every label in
$\TT$ satisfies the hypotheses of Proposition \ref{prop:collapse}. Using this
proposition, we inductively collapse every pair of leaves of $\TT$ until we
have a single vertex labeled $D_{P_k,s_k} \circ \cdots \circ D_{P_1,s_1}$, for
knots $P_1, \dots, P_k$ and integers $s_1, \dots, s_k$ with $s_i < 2\tau(P_i)$.
Thus, the knot $D_{P_k,s_k} \circ \cdots \circ D_{P_1,s_1} (K)$ is a covering
link of $Wh_+(B_T(K))$. By Theorem \ref{thm:taudjskt}, $\tau(D_{P_k,s_k} \circ
\cdots \circ D_{P_1,s_1} (K))=1$. Thus, $D_{P_k,s_k} \circ \cdots \circ
D_{P_1,s_1} (K)$ cannot be smoothly slice in a rational homology $4$-ball, so
$Wh_+(B_T(K))$ cannot be smoothly slice.

\begin{figure} \centering
\includegraphics[scale=0.8]{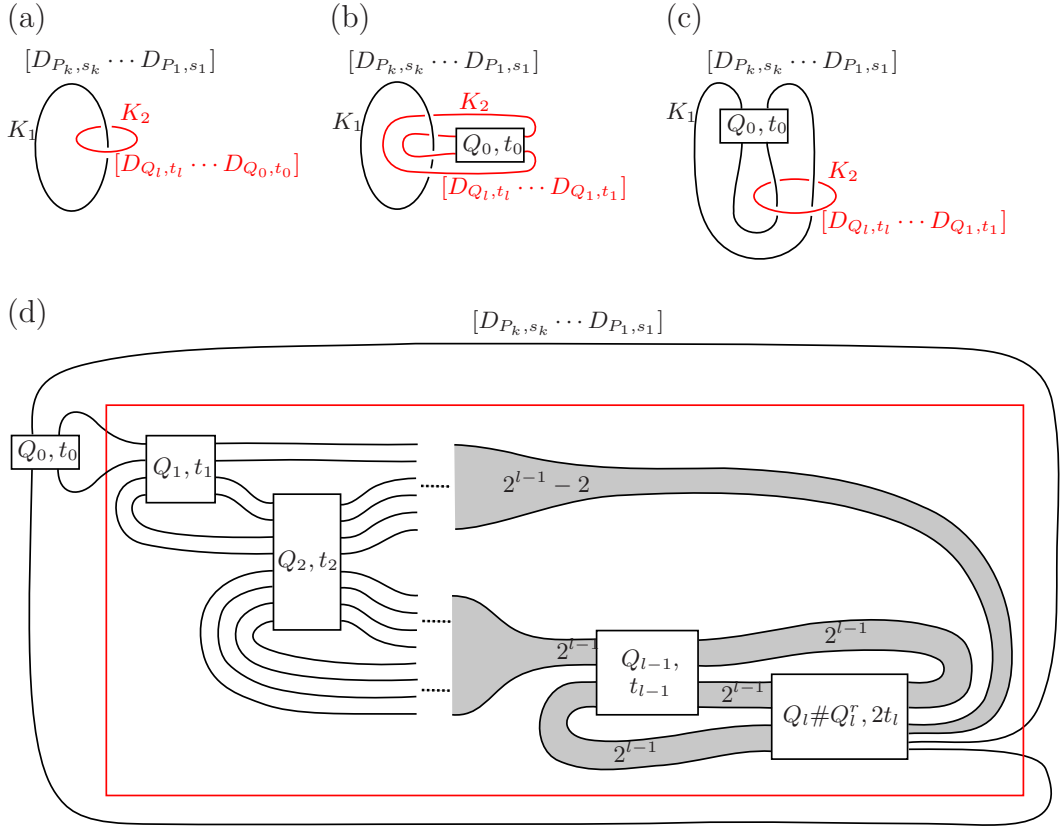}
\caption{The proof of the second part of Theorem \ref{thm:notslice}.}
\label{fig:hopf}
\end{figure}

For the second part, the same argument as above shows that by using covering
moves, we may replace $Wh_+(B_{T_1}(K_1) \cup B_{T_2}(K_2))$ with a
two-component link of the form
\[
D_{P_k,s_k} \circ \cdots \circ D_{P_1,s_1} (K_1) \cup D_{Q_l,t_l} \circ \cdots
\circ D_{Q_0,t_0} (K_2),
\]
shown in Figure \ref{fig:hopf}(a), where $s_i < 2\tau(P_i)$ and $t_i <
2\tau(Q_i)$ for all $i$. (We start with $Q_0$ and $t_0$ for notational
reasons.) After the isotopies in Figure \ref{fig:hopf}(a--c), note the
similarity to Figure \ref{fig:solidtorus1}. We may thus proceed just as in the
proof of Lemma \ref{lemma:solidtorus}, with suitable modifications to Figures
\ref{fig:solidtorus2}--\ref{fig:solidtorus4}, to obtain the knot shown in
Figure \ref{fig:hopf}(d) as a covering link of $Wh_+(B_{T_1}(K_1) \cup
B_{T_2}(K_2))$. This knot is
\[
D_{P_k,s_k} \circ \dots \circ D_{P_1,s_1} (D_{R,u} (Q_0,t_0)),
\]
where $(R,u)$ is as in \eqref{eq:solidtorus2}. This knot has $\tau=1$ by
Theorem \ref{thm:taudjskt}, completing the proof.
\end{proof}

\section{Strongly quasipositive knots and sliceness} \label{sec:quasipositive}

We conclude with a brief discussion of strongly quasipositive knots, which
played a role in an earlier version of this paper.

A knot or link $L$ is called \emph{quasipositive} if it is the closure of a
braid that is the product of conjugates of the standard positive braid
generators $\sigma_i$ (but not their inverses). It is called \emph{strongly
quasipositive} if it is the closure of a braid that is the product of words of
the form $\sigma_i \dots \sigma_{j-1} \sigma_j \sigma_{j-1}^{-1} \dots
\sigma_i^{-1}$ for $i<j$. A strongly quasipositive link naturally admits a
particular type of Seifert surface determined by this braid form, and an
embedded surface in $S^3$ is called \emph{quasipositive} if it is isotopic to
such a surface. In other words, a link is strongly quasipositive if and only if
it bounds a quasipositive Seifert surface.

A link $L$ is quasipositive if and only if it is a \emph{transverse $\C$-link}:
the transverse intersection of $S^3 \subset \C^2$ with a complex curve $V$. If
$L$ is strongly quasipositive, then the Seifert surface determined by the braid
form is isotopic to $V \cap B^4$.

For a knot $K$ and $t \in \Z$, let $A(K,t)$ be an annulus in $S^3$ whose core
circle is $K$ and whose two boundary components are $t$-framed longitudes of
the core. Given two unlinked annuli $A$ and $A'$, let $A*A'$ denote the surface
obtained by plumbing $A$ and $A'$ together. (To be precise, we must orient the
core circles of $A$ and $A'$ and specify the sign of their intersection in $A *
A'$.)

The following is a summary of some of Rudolph's results \cite{RudolphAnnuli,
RudolphObstruction, RudolphPlumbing} on strongly quasipositive knots:
\begin{thm} \label{thm:rudolph} $ \ $
\begin{enumerate}
\item If $K$ is a strongly quasipositive knot other than the unknot, then $K$
is not smoothly slice.
\item A knot $K$ is strongly quasipositive if and only if $A(K,0)$ is a
quasipositive surface.
\item If $K$ and $K'$ are strongly quasipositive, then $K \# K'$ is strongly
quasipositive.
\item The annulus $A(K,t)$ is quasipositive if and only if $t \le TB(K)$, where $TB(K)$
denotes the maximal Thurston--Bennequin number of $K$.
\item If $A$ and $A'$ are annuli, then the surface $A*A'$ is quasipositive if and only if
$A$ and $A'$ are both quasipositive.
\end{enumerate}
\end{thm}

Rudolph's original proof of (1) relies on the fact that complex curves are
genus-minimizing, a major theorem proven by Kronheimer and Mrowka
\cite{KronheimerMrowka} using gauge theory. Since a strongly quasipositive knot
$K$ has a Seifert surface that is isomorphic to a complex curve, we thus see
that $g_4(K) = g(K)$; in particular, if $K$ is nontrivial, then $g_4(K)>0$.
Subsequently, Livingston \cite{LivingstonComputations} proved that both of
these genera are equal to $\tau(K)$ when $K$ is strongly quasipositive. (For
more on the relationship between $\tau$ and quasipositivity, see Hedden
\cite{HeddenPositivity}.)

The untwisted $\pm$ Whitehead double of $K$, $Wh_{\pm}(K)$, is the boundary of
$A(K,0)*A(O,\mp1)$, where $O$ denotes the unknot. Thus, Theorem
\ref{thm:rudolph} implies that if $K$ is strongly quasipositive and nontrivial,
then $Wh_+(K)$ is strongly quasipositive and nontrivial, hence not smoothly
slice. More generally, the plumbing $A(J,s)*A(K,t)$ is a Seifert surface for
$D_{J,s}(K,t)$, so if $J$ and $K$ are strongly quasipositive and $s,t \le 0$,
then $D_{J,s}(K,t)$ is strongly quasipositive. Moreover, if neither of the
pairs $(J,s)$ and $(K,t)$ equals $(O,0)$, then $D_{J,s}(K,t)$ is nontrivial,
hence not smoothly slice. Furthermore, in this case $\tau(D_{J,s}(K,t))=1$
since the $\tau$ invariant of a strongly quasipositive knot is equal to its
genus by a result of Livingston \cite{LivingstonComputations}. Using this
observation, we may prove a weakened version of Theorem \ref{thm:notslice} in
which the knot $K$ is assumed to be strongly quasipositive without ever making
reference to Theorem \ref{thm:taudjskt}.

\bibliography{bibliography}

\providecommand{\bysame}{\leavevmode\hbox to3em{\hrulefill}\thinspace}
\providecommand{\MR}{\relax\ifhmode\unskip\space\fi MR }
\providecommand{\MRhref}[2]{%
  \href{http://www.ams.org/mathscinet-getitem?mr=#1}{#2}
}
\providecommand{\href}[2]{#2}
\begin{thebibliography}{10}

\bibitem{Bizaca}
{\v{Z}}arko Bi{\v{z}}aca, \emph{An explicit family of exotic {C}asson handles},
  Proc. Amer. Math. Soc. \textbf{123} (1995), no.~4, 1297--1302.

\bibitem{ChaKim}
Jae~Choon Cha and Taehee Kim, \emph{Covering link calculus and iterated {B}ing
  doubles}, Geom. Topol. \textbf{12} (2008), no.~4, 2173--2201.

\bibitem{ChaLivingstonRuberman}
Jae~Choon Cha, Charles Livingston, and Daniel Ruberman, \emph{Algebraic and
  {H}eegaard-{F}loer invariants of knots with slice {B}ing doubles}, Math.
  Proc. Cambridge Philos. Soc. \textbf{144} (2008), no.~2, 403--410.

\bibitem{CimasoniSlicing}
David Cimasoni, \emph{Slicing {B}ing doubles}, Algebr. Geom. Topol. \textbf{6}
  (2006), 2395--2415.

\bibitem{CochranHarveyLeidyDoubling}
Tim Cochran, Shelly Harvey, and Constance Leidy, \emph{Link concordance and
  generalized doubling operators}, Algebr. Geom. Topol. \textbf{8} (2008),
  no.~3, 1593--1646.

\bibitem{CochranOrrBoundary}
Tim Cochran and Kent Orr, \emph{Not all links are concordant to boundary
  links}, Bull. Amer. Math. Soc. (N.S.) \textbf{23} (1990), no.~1, 99--106.

\bibitem{FreedmanNewTechnique}
Michael~H. Freedman, \emph{A new technique for the link slice problem}, Invent.
  Math. \textbf{80} (1985), no.~3, 453--465.

\bibitem{FreedmanWhitehead3}
\bysame, \emph{{${\rm Whitehead}_3$} is a ``slice'' link}, Invent. Math.
  \textbf{94} (1988), no.~1, 175--182.

\bibitem{FreedmanQuinn}
Michael~H. Freedman and Frank Quinn, \emph{Topology of 4-manifolds}, Princeton
  Mathematical Series, vol.~39, Princeton University Press, Princeton, NJ,
  1990.

\bibitem{Ghiggini}
Paolo Ghiggini, \emph{Knot {F}loer homology detects genus-one fibred knots},
  Amer. J. Math. \textbf{130} (2008), no.~5, 1151--1169.

\bibitem{HeddenPositivity}
Matthew Hedden, \emph{Notions of positivity and the {O}zsv{\'a}th--{S}zab{\'o}
  concordance invariant}, to appear, \arxiv{math/0509499}.

\bibitem{HeddenWhitehead}
\bysame, \emph{Knot {F}loer homology of {W}hitehead doubles}, Geom. Topol.
  \textbf{11} (2007), 2277--2338.

\bibitem{KauffmanOnKnots}
Louis~H. Kauffman, \emph{On knots}, Annals of Mathematics Studies, vol. 115,
  Princeton University Press, Princeton, NJ, 1987.

\bibitem{KirbyList}
Rob Kirby (ed.), \emph{Problems in low-dimensional topology}, AMS/IP Stud. Adv.
  Math., vol.~2, Amer. Math. Soc., Providence, RI, 1997.

\bibitem{KronheimerMrowka}
Peter Kronheimer and Tomasz Mrowka, \emph{Gauge theory for embedded surfaces.
  {I}}, Topology \textbf{32} (1993), no.~4, 773--826.

\bibitem{LevineDoublingOperators}
Adam~S. Levine, \emph{Knot doubling operators and bordered {H}eegaard {F}loer
  homology}, preprint (2010), \arxiv{1008.3349}.

\bibitem{LOTBimodules}
Robert Lipshitz, Peter Ozsv{\'a}th, and Dylan Thurston, \emph{Bimodules in
  bordered {H}eegaard {F}loer homology}, preprint (2010), \arxiv{1003.0598}.

\bibitem{LOTBordered}
\bysame, \emph{Bordered {H}eegaard {F}loer homology: invariance and pairing},
  preprint (2009), \arxiv{0810.0687}.

\bibitem{LivingstonComputations}
Charles Livingston, \emph{Computations of the {O}zsv\'ath-{S}zab\'o knot
  concordance invariant}, Geom. Topol. \textbf{8} (2004), 735--742
  (electronic).

\bibitem{NiFibered}
Yi~Ni, \emph{Knot {F}loer homology detects fibered knots}, Invent. Math.
  \textbf{170} (2007), no.~3, 577--608.

\bibitem{OSz4Genus}
Peter Ozsv{\'a}th and Zolt{\'a}n Szab{\'o}, \emph{Knot {F}loer homology and the
  four-ball genus}, Geom. Topol. \textbf{7} (2003), 615--639 (electronic).

\bibitem{OSzGenus}
\bysame, \emph{Holomorphic disks and genus bounds}, Geom. Topol. \textbf{8}
  (2004), 311--334 (electronic).

\bibitem{OSzKnot}
\bysame, \emph{Holomorphic disks and knot invariants}, Adv. Math. \textbf{186}
  (2004), no.~1, 58--116.

\bibitem{OSz3Manifold}
\bysame, \emph{Holomorphic disks and topological invariants for closed
  three-manifolds}, Ann. of Math. (2) \textbf{159} (2004), no.~3, 1027--1158.

\bibitem{OSz4Manifold}
\bysame, \emph{Holomorphic triangles and invariants for smooth four-manifolds},
  Adv. Math. \textbf{202} (2006), no.~2, 326--400.

\bibitem{RasmussenThesis}
Jacob~A. Rasmussen, \emph{Floer homology and knot complements}, Ph.D. thesis,
  Harvard University, 2003, arXiv:math/0509499.

\bibitem{RudolphAnnuli}
Lee Rudolph, \emph{Quasipositive annuli. ({C}onstructions of quasipositive
  knots and links. {IV})}, J. Knot Theory Ramifications \textbf{1} (1992),
  no.~4, 451--466.

\bibitem{RudolphObstruction}
\bysame, \emph{Quasipositivity as an obstruction to sliceness}, Bull. Amer.
  Math. Soc. (N.S.) \textbf{29} (1993), no.~1, 51--59.

\bibitem{RudolphPlumbing}
\bysame, \emph{Quasipositive plumbing (constructions of quasipositive knots and
  links. {V})}, Proc. Amer. Math. Soc. \textbf{126} (1998), no.~1, 257--267.

\bibitem{VanCott}
Cornelia Van~Cott, \emph{An obstruction to slicing iterated {B}ing doubles},
  preprint (2009), \arxiv{math/0907.4948}.

\end{thebibliography}
\bibliographystyle{amsplain}

\end{document}